\documentclass{amsart}
\usepackage{mathpazo}
\usepackage{amssymb}
\usepackage{newtxmath}
\newtheorem{theorem}{Theorem}[section]
\newtheorem{lemma}[theorem]{Lemma}
\newtheorem*{Acknowledgment}{\textnormal{\textbf{Acknowledgment}}}
\theoremstyle{definition}
\newtheorem{definition}[theorem]{Definition}

\newtheorem{corollary}[theorem]{Corollary}

\numberwithin{equation}{section}
\newcommand{\beqa}{\begin{eqnarray*}}
\newcommand{\eeqa}{\end{eqnarray*}}
\newcommand{\beqn}{\begin{eqnarray}}
\newcommand{\eeqn}{\end{eqnarray}}

\renewcommand{\a}{\alpha}

\newcounter{cnt1}
\newcounter{cnt2}
\newcounter{cnt3}
\newcommand{\blr}{\begin{list}{$($\roman{cnt1}$)$}
        {\usecounter{cnt1} \setlength{\topsep}{0pt}
                \setlength{\itemsep}{0pt}}}
\newcommand{\bla}{\begin{list}{$($\alph{cnt2}$)$}
        {\usecounter{cnt2} \setlength{\topsep}{0pt}
                \setlength{\itemsep}{0pt}}}
\newcommand{\bln}{\begin{list}{$($\arabic{cnt3}$)$}
        {\usecounter{cnt3} \setlength{\topsep}{0pt}
                \setlength{\itemsep}{0pt}}}
\newcommand{\el}{\end{list}}
\newtheorem{thm}{Theorem}

\newtheorem{Def}[thm]{Definition}

\newtheorem{rem}[thm]{Remark}
\newcommand{\Rem}{\begin{rem} \rm}
\newcommand{\bdfn}{\begin{Def} \rm}
\newcommand{\edfn}{\end{Def}}

\title{Two aspects of small diameter properties} 
\author[ S. Basu , S. Seal ]
	{Sudeshna Basu$^{1}$, Susmita Seal$ ^{2}$ }
\address{{$^{1}$}   Sudeshna Basu,
		Department of Mathematics and Statistics, 
				Loyola University, 
				Baltimore, MD 21210, USA 
		}
	\email{sudeshnamelody@gmail.com}

\address {{$^{2}$} Susmita Seal, 
		Department of Mathematics, ,
		Ramakrishna Mission Vivekananda Education and Research Institute , 
		Belur Math,  Howrah 711202,
		West Bengal, India}
	\email{susmitaseal1996@gmail.com}	

\date{}
\begin{document}
\maketitle
\begin{abstract}
In this short note, we study two different geometrical aspects of Banach spaces with  small diameter properties, namely the Ball Dentable Property ($BDP$), Ball Huskable Property ($BHP$) and Ball Small Combination of slice Property ($BSCSP$). We show that $BDP$, $BHP$ and $BSCSP$ are separably determined properties.
We also  explore the stability of these properties  over K$\ddot{o}$the Bochner spaces.

\end{abstract}
\section{Introduction}
Throughout this work, we consider $X$, a {\it real} Banach space.
 We denote by $X^*$, the  dual of $X.$  We denote the closed unit ball with center at $0$, by $B_X $, the unit sphere by $S_X$ and the closed ball of radius $r >0$ and center $x$ by $B_X(x, r).$
  We refer to the monograph \cite{Br} for notions of convexity theory that we will be using here. Let $X$ be a Banach space. A slice of a bounded set $C\subset X$ is defined  by 
$$S(C, x^*, \a) = \{x \in C : x^*(x) > \mbox{sup}~ x^*(C) - \a \}$$
where $x^*\in X^*$ and $\a > 0.$
 We assume without loss of generality that $\|x^*\| = 1$. Analogously one can define $w^*$-slices in $X^*$. 
 \begin{definition}
 	\cite{BS}
 	A Banach space $X$ has 
 	\begin{enumerate}
 		\item   Ball Dentable Property ($BDP$) if $B_X$ has  
 		slices of arbitrarily small diameter. 
 		\item   Ball Huskable Property ($BHP$) if $B_X$ 
 		has nonempty relatively weakly open subsets of arbitrarily small diameter.
 		\item   Ball Small Combination of Slice Property ($BSCSP$) 
 		if  $B_X$ has convex combination of slices of arbitrarily 
 		small diameter.
 	\end{enumerate}
\end{definition}
One can analogously define $w^*$-$BDP$, $w^{*}$-$BHP$ and $w^{*}$-$BSCSP$ in a dual space where slices and weakly open sets are replaced by $w^*$-slices and $w^*$-open sets  respectively.
We have the following implications. In general, none of the reverse implications hold good. For details, see \cite{BS}.
$$ BDP \  \ \Longrightarrow \ BHP \ \Longrightarrow \  \ BSCSP$$ $   \quad \quad \quad\quad\quad\quad\quad\quad\quad\quad\quad \Big \Uparrow \quad \quad\quad\quad\quad \Big \Uparrow \quad \quad\quad\quad\quad \Big \Uparrow$  $$ w^*-BDP \Longrightarrow  w^*-BHP \Longrightarrow  w^*-BSCSP$$

Banach spaces with Radon Nikodym Property ($RNP$) has been an important area of research in geometry of Banach spaces. The reason being, Banach spaces with $RNP$ has nice characterizations in terms of geometry of the underlying space. For more details, see \cite{Br}, \cite{DU}. One equivalent characterization of a Banach space  $X$ with $RNP$ is that that every closed bounded convex subset of $X$ has slices with arbitrarily small diameter. Two closely related properties are the Point of Continuity Property ($PCP$), namely,  every closed bounded subset of $X$  has nonempty relatively weakly open subsets with arbitrarily small diameter and Strong Regularity ($SR$) namely,  every closed, convex and bounded subset of $X$ has convex combination of slices with  arbitrarily small diameter. For more details, see \cite{Bj}, \cite{GGMS} and \cite{GMS}. A slice of $B_X$ is a weakly open subset of $B_X$. Also Bourgain's Lemma \cite{Ro} tells us any nonempty weakly open subset of a bounded convex set $C$ contains a convex combination of slices of $C$.  It is clear from the definitions above, that  $RNP$ implies $BDP$, $PCP$ implies $BHP$ and $SR$ implies $BSCSP$. So we have,
 \newpage
 $$RNP \ \Longrightarrow \  PCP  \ \Longrightarrow  \ SR$$ $   \quad\quad\quad \quad\quad\quad\quad\quad\quad\quad\quad \quad \quad \Big \Downarrow   \quad \quad\quad\quad \Big \Downarrow  \quad \quad \quad\quad \Big \Downarrow$  $$ \quad BDP \ \Longrightarrow \  BHP \ \Longrightarrow  \ BSCSP$$
 So, in a sense these  three properties namely the $BDP, BHP$ and $BSCSP$ are the localized version (to the closed unit ball) of the 
 $RNP, PCP$ and  $SR.$ The small diameter properties  has nice geometric properties in terms of stability with respect to $ l_p$ sums ($(1\leqslant p < \infty )$) of Banach spaces and Lebesgue-Bochner spaces.
 They satisfy  the three space property  under certain assumption on the quotient space. The small diameter properties are inherited by certain ideals of a Banach space. Also, if the duals of certain ideals of a Banach space  have the $w^*$-version of these properties, then the dual of the Banach space also has these properties.  For more details, see \cite{BR}, \cite{Ba},  \cite{BS} and \cite{SBBGVY}.

 It is well known that all the properties $RNP$, $PCP$ and $SR$ are hereditary. 
However, that is not true for $BDP$, $BHP$ and $BSCSP$. For example,  $l_1\oplus_1 c_0$ has $BDP$, $BHP$ and $BSCSP$ \cite{BS}, but its subspace $c_0$ which in particular is also separable,  does not have  $BDP$, $BHP$ and $BSCSP$ \cite{La}. The next best thing we can expect is that the spaces with small diameter properties will satisfy the separably determined property.
\begin{definition}
A property (P) in non separable Banach space $X$ is called a separably determined property for $X$ if for every separable subspace $Z$ in $X,$ there exists a separable subspace $Z'$ in $X$ such that $Z\subset Z'$ and $Z'$ has the property (P) in $X.$ 
\end{definition}
In this work, we prove that all these three small diameter properties are separably determined properties.

\begin{definition}
\cite{Li} Let $(\Omega,\mathcal{A},\mu)$ be a complete, $\sigma$ finite measure space. Then  K$\ddot{o}$the function space over $(\Omega,\mathcal{A},\mu)$ is a Banach space $(E,\|.\|_E)$ of real valued measurable functions on 
$\Omega$ modulo equality $\mu$ a.e. satisfying 
\begin{enumerate}
\item $\chi_A \in E$ for all $A\in \mathcal{A}$ with $\mu(A)<\infty$
\item f is $\mu$ integrable over $A$ for all $f\in E$, for all $A\in \mathcal{A}$ with $\mu(A)<\infty.$ 
\item if g is measurable, $f\in E$ such that $|g(t)|\leqslant |f(t)|$ $\mu$ a.e. then $g\in E$ and $\|g\|_E\leqslant \|f\|_E$
\end{enumerate}
\end{definition}

A function $f:\Omega \rightarrow X$ is called simple if $f=\sum_{i=1}^{n} x_i \chi_{A_i}$ where $ A_i \in \mathcal{A}$ are disjoint with $\mu (A_i)<\infty$ and $ 0 \neq x_i  \in X$ $\forall i=1,2,\ldots,n.$ If there exists a sequence of simple functions $
\{f_n\}$ such that $\lim \limits_{n} \|f_n(t)-f(t)\|=0$ $\mu$  a.e., then $f$ is called Bochner measurable.

\begin{definition}
\cite{Li} For every K$\ddot{o}$the function space $E$ over complete, $\sigma$ finite measure space $(\Omega,\mathcal{A},\mu)$ and Banach space $X$, the  K$\ddot{o}$the-Bochner space is given by  
$$ E(X)=\{ f : f:\Omega \rightarrow X \ \mathrm{is} \ \mathrm{Bochner} \ \mathrm{measurable} \ \mathrm{and} \  \|.\| \circ f\in E\}$$ is a Banach space with norm $\|f\|_{E(X)} = \| (\|.\| \circ f) \|_E.$
 \end{definition}
 An easy example of such spaces is Lebesgue-Bochner space. For more details, see \cite{Li}.
In this work we explore  the stability of the above mentioned small diameter properties for  K$\ddot{o}$the-Bochner spaces.
We prove that $X$ has $BDP$ ($BHP$) if the K$\ddot{o}$the-Bochner space has $BDP$ ($BHP$). The question is still open for $BSCSP$.

\section{Main results}
We first prove that the small diameter properties are all separably determined property.
We need the following definition and result for our discussion.
\begin{definition}
	\cite{LO}
	Let $X$ be a Banach space and $Y$ be a subspace of $X.$
	An operator $f : Y^* \rightarrow X^*$ is called a  Hahn Banach extension operator  if it satisfies following two conditions : 
	\begin{enumerate}
		\item $(fy^*)|_Y = y^*$ $\forall y^* \in Y^*$
		\item $\Vert fy^* \Vert = \Vert y^* \Vert $ $\forall y^* \in Y^*$
	\end{enumerate}
\end{definition}

\begin{lemma} \label{sep lem}
	\cite[Lemma 4.3]{HWW}
	If $X$ is a Banach space, $L$ a separable subspace of $X,$ and $F$ a separable subspace of $X^*,$ then $X$ has a separable subspace $M$ containing $L$ which admits a Hahn-Banach extension operator $T:M^*\rightarrow X^*$ satisfying $TM^*\supset F.$
\end{lemma}


\begin{theorem}
	$BSCSP$ is a separably determined property.
\end{theorem}

\begin{proof}
Let $Y$ be any separable subspace of $X.$ Since $X$ has $BSCSP,$  for each $n\in \mathbb{N}$ there exists
 $k(n)\in \mathbb{N}$
such that $$\mathrm{diam} \  \sum_{i=1}^{k(n)} \lambda_{n,i} \ S(B_X,x_{n,i}^*,\alpha_{n,i}) <\frac{1}{n}$$
where $x_{n,1}^*, x_{n,2}^*, \ldots, x_{n,k(n)}^* \in S_{X^*}$, $\alpha_{n,1}, \alpha_{n,2}, \ldots, \alpha_{n,k(n)} \in \mathbb{R}^+$ and $\lambda_{n,1}, \lambda_{n,2}, \ldots, \lambda_{n,k(n)} >0$ with $\sum_{i=1}^{k(n)} \lambda_{n,i} =1.$
Consider separable subspace $F$= $\overline{span}$ $\bigcup_{n=1}^{\infty} \{x_{n,1}^*, x_{n,2}^*, \ldots x_{n,k(n)}^* \}$ in $X^*$. Then, by Lemma $\ref{sep lem}$ there exists a separable subspace $Z$ in $X$ such that $Y\subset Z$ and a Hahn-Banach extension operator $T:Z^*\rightarrow X^*$ such that $F\subset T(Z^*).$ Now we will show that $Z$ has $BSCSP$. So let $\varepsilon >0$. Choose $m\in \mathbb{N}$ such that $\frac{1}{m}<\varepsilon$.
 Then for each $i=1,2,\ldots, k(m)$ we have $x_{m,i}^*= T(z_{m,i}^*)$ for some $z_{m,i}^*\in S_{Z^*}$.
  Consider convex combination of slices $\sum_{i=1}^{k(m)} \lambda_{m,i} \ S(B_Z,z_{m,i}^*,\alpha_{m,i})$ of $B_Z.$
  Observe that
  \begin{equation}\notag
\begin{split}
\sum _{i=1}^{k(m)} \lambda_{m,i} \ S(B_Z,z_{m,i}^*,\alpha_{m,i})= \sum _{i=1}^{k(m)} \lambda_{m,i} \ \{ z\in B_Z : z_{m,i}^*(z)>1-\alpha_{m,i}\}\\
\hspace{3 cm}=\sum _{i=1}^{k(m)} \lambda_{m,i} \ \{ z\in B_Z : x_{m,i}^*(z)>1-\alpha_{m,i}\}\\
\subset \sum _{i=1}^{k(m)} \lambda_{m,i} \ S(B_X,x_{m,i}^*,\alpha_{m,i})\hspace{1.8 cm}
\end{split}
\end{equation}
  Thus, diam $\sum _{i=1}^{k(m)} \lambda_{m,i} \ S(B_Z,z_{m,i}^*,\alpha_{m,i}) \leqslant$ diam $\sum _{i=1}^{k(m)} \lambda_{m,i} \ S(B_X,x_{m,i}^*,\alpha_{m,i})< \frac{1}{m} <\varepsilon.$
Hence $Z$ has $BSCSP.$
\end{proof}

Arguing similarly as above, we have the following Corollary. 
\begin{corollary}
 $BDP$ is a separably determined property.
\end{corollary} 

\begin{theorem}
 $BHP$ is a separably determined property.
\end{theorem}

\begin{proof}
Let $Y$ be any separable subspace of $X.$  Since $X$ has $BHP,$ for each $n\in \mathbb{N}$ there exists a nonempty weakly open set $$W_n =\{ x\in B_X : |x_{n,i}^*(x-x_n)|<\delta_n ,\  i=1,\ldots, k(n)\}$$ 
of $B_X$ with diameter less than $\frac{1}{n},$ where $x_n\in B_X$ and $x_{n,1}^*,x_{n,2}^*,\ldots,x_{n,k(n)}^* \in X^*.$
 Consider separable subspace $F$= $\overline{span}$ $\bigcup_{n=1}^{\infty} \{x_{n,1}^*, x_{n,2}^*, \ldots x_{n,k(n)}^* \}$ in $X^*$.
  Then by Lemma $\ref{sep lem}$
there exists a separable subspace $Z$ in $X$ such that $Y\subset Z$ and a Hahn-Banach extension operator $T:Z^*\rightarrow X^*$ such that $F\subset T(Z^*).$ Now we will show that $Z$ has $BHP$. So let $\varepsilon >0$. Choose $m\in \mathbb{N}$ such that $\frac{1}{m}<\varepsilon$.
 Then for each $i=1,2,\ldots, k(m)$ we have
 $x_{m,i}^*= T(z_{m,i}^*)$ for some $z_{m,i}^*\in Z^*.$ 
Also $T^*x_m\in B_{Z^{**}}.$ Choose $\alpha_m>0$ such that $2\alpha_m<\delta_m.$
Since $B_Z$ is $w^*$ dense in $B_{Z^{**}},$ then 
$$\{ z^{**}\in B_{Z^{**}} : |z_{m,i}^*(z^{**}-T^*x_m)|<\alpha_m, i=1,\ldots, k(m)\} \bigcap B_Z \neq \emptyset$$
Choose $z_m$ from above intersection. Then $z_m\in B_Z$ and
\begin{equation}
|z_{m,i}^*(z_m)-T^*x_m(z_{m,i}^*)|<\alpha_m \quad \forall i=1,2,\ldots,k(m).
\end{equation}
 which gives 
 \begin{equation}\notag
 \begin{split}
 |x_{m,i}^*(z_m)-x_{m,i}^*(x_m)|=|x_{m,i}^*(z_m)-x_m(x_{m,i}^*)|
 =|z_{m,i}^*(z_m)-x_m(T z_{m,i}^*)|\hspace{2 cm}\\
 =|z_{m,i}^*(z_m)-T^*x_m(z_{m,i}^*)|\hspace{1.8 cm}\\
 <\alpha_m \quad \forall i=1,2,\ldots,k(m).\hspace{1.6 cm}
  \end{split}
 \end{equation}
Therefore consider weakly open set 
$$V_m =\{ z\in B_Z : |z_{m,i}^*(z-z_m)|<\alpha_m , i=1,\ldots, k(m)\}$$ 
 of $B_Z.$ Then $V_m\subset W_m.$
 Indeed, let $z\in V_m,$ then for each $i=1,2,\ldots ,k(m)$,
 \begin{equation}\notag
 \begin{split}
|x_{m,i}^*(z)-x_{m,i}^*(x_m)|\leqslant |x_{m,i}^*(z)-x_{m,i}^*(z_m)|+ |x_{m,i}^*(z_m)-x_{m,i}^*(x_m)|\\
= |z_{m,i}^*(z)-z_{m,i}^*(z_m)|+ |x_{m,i}^*(z_m)-x_{m,i}^*(x_m)|\\
<\delta_m.\hspace{5.8 cm} 
\end{split}
 \end{equation}
Therefore, diam $(V_m)\leqslant $ diam $(W_m)<\frac{1}{m}<\varepsilon.$
Hence $Z$ has $BHP.$
\end{proof}

Next, we explore the stability results of small diameter properties over K$\ddot{o}$the-Bochner spaces. We use some techniques which are similar to the ones used in \cite{LP}. 

\begin{theorem} \label{bochner bhp}
Suppose $(\Omega,\mathcal{A},\mu)$ is a complete, $\sigma$ finite measure space. Let $E$ be a K$\ddot{o}$the function space over $(\Omega,\mathcal{A},\mu)$ and  $X$ be a Banach space such that the simple functions are dense in $E(X).$ Suppose  $E(X)$ has $BHP.$ Then $X$ has $BHP.$ 
\end{theorem}
\begin{proof}
We prove by contradiction.
Suppose $X$ does not have $BHP.$ Then there exists $\varepsilon >0$ such that the diameter of any nonempty relatively weakly open subset of $B_X$ is greater than $\varepsilon.$ Since $E(X)$ has $BHP,$ 
there exist  slices $S(B_{E(X)},f_i^*,\alpha_i)$  $(1\leqslant i \leqslant n)$ of $B_{E(X)}$ such that $\bigcap_{i=1}^{n} S(B_{E(X)},f_i^*,\alpha_i)\neq \emptyset$ and diameter of $\bigcap_{i=1}^{n} S(B_{E(X)},f_i^*,\alpha_i) $ is less than $\varepsilon.$
 Choose $f\in \bigcap_{i=1}^{n} S(B_{E(X)},f_i^*,\alpha_i)$ with $\|f\|_{E(X)}=1.$ Since simple functions are dense in $E(X)$, without loss of generality we can assume that 
 $$f=\sum_{k=1}^{m} x_k \chi_{A_k}$$ 
 where $A_k$ are disjoint sets in $\mathcal{A}$ and $x_k\neq 0$  for all $k=1,2,\ldots,m.$ For each $i=1,\ldots,n$ and $k=1,\ldots,m,$ we define continuous, linear map $x_{ik}^*(x)=f_i^*(x\chi_{A_k}).$
 For each k, consider relatively weakly open subset 
 $$W_k= \bigcap_{i=1}^{n} S(\|x_k\|B_X,x_{ik}^*,\|x_{ik}^*\| \|x_{k}\|-x_{ik}^*(x_{k})+\alpha_{ik})$$ 
 of $\|x_k\|B_X,$ where $\alpha_{ik}>0$ are such that $\sum_{k=1}^{m} \alpha_{ik} \leqslant f_i^*(f)-(1-\alpha_{i}).$  Observe that $x_k\in W_k,$ and hence $W_k\neq \emptyset.$ Since any nonempty relatively weakly open subset of $B_X$ has diameter greater than $\varepsilon,$ then we claim that diameter of $W_k$ is greater than $\|x_k\| \varepsilon.$ 
 Indeed, consider $V_k= \bigcap_{i=1}^{n} S(B_X,x_{ik}^*,\frac{\|x_{ik}^*\| \|x_{k}\|-x_{ik}^*(x_{k})+\alpha_{ik}}{\|x_k\|}).$ Observe $\frac{x_k}{\|x_k\|}\in V_k.$ Thus $V_k$ is a nonempty weakly open subset of $B_X$ and hence by our assumption diam $(V_k)>\varepsilon.$ Therefore, there exist $y_{k1},y_{k2}\in V_k$ such that $\|y_{k1} - y_{k2}\|>\varepsilon.$ Observe that 
\begin{equation} \label{kothe 1}
\begin{split}
x_{ik}^*(\|x_k\| y_{k1})= \|x_k\| x_{ik}^*(y_{k1})>\|x_k\| \Big(\|x_{ik}^*\| - \frac{\|x_{ik}^*\| \|x_{k}\|-x_{ik}^*(x_{k})+\alpha_{ik}}{\|x_k\|}\Big)\\
 = x_{ik}^*(x_k)-\alpha_{ik} \quad \forall i=1,\ldots,n. \hspace{1.8 cm}
\end{split}
\end{equation} 
and 
\begin{equation}\label{kothe 2}
\Big\| \|x_k\|y_{k1} - \|x_k\|y_{k2} \Big\|=\|x_k\|  \|y_{k1}-y_{k2}\|>\|x_k\| \varepsilon.
\end{equation} 
Thus, combining $(\ref{kothe 1})$ and $(\ref{kothe 2})$ we have $\|x_k\|y_{k1},\|x_k\|y_{k2} \in W_k$ and $\|$ $\|x_k\|y_{k1} - \|x_k\|y_{k2}$ $\| > \|x_k\| \varepsilon.$ Hence, our claim is true.
 Therefore, there exist $y_k,z_k\in W_k$ such that $\|y_k-z_k\|>\|x_k\|\varepsilon.$ Define $$g=\sum_{k=1}^{m} y_k \chi_{A_k}$$ 
  $$h=\sum_{k=1}^{m} z_k \chi_{A_k}.$$
  
 Observe that,
 since $y_k\in W_k,$ then $\|y_k\|\leqslant \|x_k\|$ $\forall k=1,\ldots,m.$ 
Therefore 
$$\|g\|_{E(X)} = \Big\| \|y_k\| \chi_{A_k} \Big\|_E \leqslant \Big\| \|x_k\| \chi_{A_k} \Big\|_E
= \|f\|_{E(X)}=1.$$ 
Also for each $i=1,\ldots,n$ 
  \begin{equation}\notag
  \begin{split}
  f_i^*(g)
=\sum_{k=1}^{m} x_{ik}^*(y_k) 
>\sum_{k=1}^{m} (x_{ik}^*(x_k)-\alpha_{ik})
=\sum_{k=1}^{m} x_{ik}^*(x_k)-\sum_{k=1}^{m} \alpha_{ik} \hspace{4.5 cm} \\
= \sum_{k=1}^{m} f_i^*(x_k\chi_{A_k}) - \sum_{k=1}^{m} \alpha_{ik} 
 = f_i^*(f)-\sum_{k=1}^{m} \alpha_{ik} 
 \geqslant 1-\alpha_i 
  \end{split}  
  \end{equation}
  We prove similarly for $h.$ Hence, $g,h \in \bigcap_{i=1}^{n} S(B_{E(X)},f_i^*,\alpha_i).$
Finally,  $$\|g-h\|_{E(X)}= \| \sum_{k=1}^{m}( y_k - z_k) \chi_{A_k}\| 
  = \Big\| \|y_k - z_k\| \chi_{A_k} \Big\|_E
  >\Big\| \|x_k\|\varepsilon \chi_{A_k} \Big\|_E
  = \varepsilon \|f\|_{E(X)} = \varepsilon$$
Thus, diam $\bigcap_{i=1}^{n} S(B_{E(X)},f_i^*,\alpha_i)\geqslant \varepsilon,$ a contradiction.
Hence, $X$ has $BHP.$
\end{proof}  


Arguing similarly as above for $n=1$, we have the following. 
\begin{corollary} \label{bochner bdp}
Suppose $(\Omega,\mathcal{A},\mu)$ is a complete, $\sigma$ finite measure space. Let $E$ be a K$\ddot{o}$the function space over $(\Omega,\mathcal{A},\mu)$ and $X$ be a Banach space such that simple functions are dense in $E(X).$ Also let $E(X)$ have $BDP.$ Then $X$ has $BDP.$
\end{corollary}

{\bf Question~:}
Suppose $(\Omega,\mathcal{A},\mu)$ is a complete, $\sigma$ finite measure space. Let $E$ be a K$\ddot{o}$the function space over $(\Omega,\mathcal{A},\mu)$ and $X$ be a Banach space such that simple functions are dense in $E(X).$ Also let $E(X)$ have $BSCSP.$ Will that imply $X$ has  $BSCSP$?

\begin{Acknowledgment}
	This work was done when the first author was visiting the Department of Pure Mathematics, University of Calcutta. She is particularly grateful to Professor Sunil Kumar Maity for his support  and hospitality.
 The second  author's research is funded by the National Board for Higher Mathematics (NBHM), Department of Atomic Energy (DAE), Government of India, Ref No: 0203/11/2019-R$\&$D-II/9249.
\end{Acknowledgment}


\begin{thebibliography}{99}
\small
\bibitem[Br]{Br} R.\ D.\ Bourgin, {\it Geometric aspects of convex sets with the Radon-Nikodym property},  Lecture Notes in Mathematics Springer-Verlag  Berlin, {\bf 993} (1983).

\bibitem[Ba]{Ba} S.\ Basu; {\it On Ball dentable property  in Banach Spaces},  Math.\ Analysis and its Applications in Modeling (ICMAAM 2018) Springer Proceedings in Mathematics and Statistics, {\bf 302} 145-149 (2020).

\bibitem[Bj]{Bj} J.\ Bourgain, {\it La propriété de Radon-Nikodym}, Publ. Mathé. de l'Université Pierre et Marie Curie, {\bf 36} (1979).



\bibitem[BR]{BR} S.\ Basu, T.\ S.\ S.\ R.\ K.\ Rao; { \it On Small Combination of slices in Banach Spaces}, Extracta Mathematica,  {\bf 31} 1-10 (2016).

\bibitem[BS]{BS} S. \ Basu, S. \ Seal, {\it Small Combination of Slices, Dentability and Stability Results Of Small Diameter Properties In Banach Spaces}, J. Math. Anal. Appl., {\bf 507} (2) (2022),  
\\ online access,  https://doi.org/10.1016/j.jmaa.2021.125793.

\bibitem[BS2]{BS2}   S.\  Basu, S. \ Seal, {\it Small combination of slices and dentability in ideals of Banach spaces}, J. Conv. Anal., { \bf30}  
(1) (2023) (To appear).



\bibitem [DU]{DU} J. \ Diestel, J.\ J. \ Uhl, Vector Measures, Amer.\  Math.\  Soc., { \bf 15}, (1977).

\bibitem [GGMS]{GGMS} N.\ Ghoussoub, G.\ Godefroy, B.\ Maurey, W.\ Scachermayer, { \it Some topological and geometrical structures in Banach spcaes}, Mem. Amer. Math. Soc., {\bf 70}  378  (1987).


\bibitem [GMS]{GMS} N.\ Ghoussoub, B.\ Maurey, W.\  Schachermayer, {\it Geometrical implications of
certain infinite-dimensional decomposition}, Trans. Amer. Math. Soc.,  {\bf 317} 541–584 (1990).



\bibitem [HWW]{HWW} P.\ Harmand, D.\ Werner, W.\ Werner, {\it $M$-ideals in Banach spaces and Banach algebras}, Lecture Notes in Mathematics, Springer-Verlag, Berlin, {\bf 1547} (1993). 





\bibitem[La]{La} J.\ Langemets, {\it Geometrical structure in diameter 2 Banach spaces}, Dissertationes Mathematicae Universitatis Tartuensis, {\bf 99} (2015).

\bibitem[Li]{Li} P.\ K.\ Lin, {\it K$\ddot{o}$the-Bochner function spaces}, Birkh$\ddot{a}$user, Boston, USA (2004).

\bibitem[LO]{LO} A.\ Lima, E.\ Oja, {\it Hahn-Banach extension operators and spaces of operators}, Proc. Amer. Math. Soc., {\bf 130} (12) 3631-3640 (2002).

\bibitem[LP]{LP} J.\ Langemets, K.\ Pirk, {\it Stability of diametral diameter two properties}, Revista de la Real Academia de ciencias Exactas, F$\acute{i}$sicas y Naturales. Series A. Matem$\acute{a}$ticas, {\bf 115} (2021). 




\bibitem[Ro]{Ro} H.\ P.\ Rosenthal, {\it On the structure of non-dentable closed bounded convex sets}, Adv. in Math.,  { \bf 70} (1) 1-58 (1988).


\bibitem[Sc]{Sc} W.\ Schachermayer,  {\it The Radon Nikodym Property and the Krein-Milman Property are equivalent for strongly regular sets}, Trans. Amer. Math. Soc., {\bf 303} (2) 673-687 (1987).

\bibitem[SBBGVY]{SBBGVY} S. Seal, S. Basu, J. Becerra Guerrero, J. M. Villegas Yeguas, { \it Non-rough norms and dentability in spaces of operators} (Submitted), available at arXiv:2210.05887 2023.

\bibitem[SSW]{SSW} W.\ Schachermayer, A.\ Sersouri, E.\ Werner, {\it Moduli of nondentability and the Radon Nikodym Property in Banach spaces}, Israel J. Math, {\bf 65} (3) 225-257 (1989).
\end{thebibliography}
\end{document}